\documentclass[12pt,final]{amsart}
\usepackage{amsmath,amssymb}

\numberwithin{equation}{section}

\usepackage{color}
\usepackage{soul,graphicx}

\usepackage[color]{showkeys}
\definecolor{refkey}{rgb}{0,0,1}
\definecolor{labelkey}{rgb}{1,0,0}

\newcommand{\ol}{\overline}

\newcommand{\calS}{\mathcal{S}}
\newcommand{\calU}{\mathcal{U}}
\newcommand{\calP}{\mathcal{P}}

\newcommand{\calA}{\mathcal{A}}
\newcommand{\calF}{\mathcal{F}}

\newcommand{\calH}{\mathcal{H}}

\newcommand{\al}{\alpha}
\newcommand{\bt}{\beta}
\newcommand{\ep}{\varepsilon}

\newcommand{\gm}{\gamma}

\newcommand{\eq} [1] {\begin{equation}\label{#1}\quad}
\newcommand{\en} {\end{equation}}

\newcommand{\adj}{\operatorname{adj}}

\newcommand{\tm}{\times}

\newcommand{\sbs}{\subset}

\newcommand{\wdt}{\widetilde}

\newcommand{\iy}{\infty}

\newcommand{\bR}{\mathbb{R}}

\newcommand{\bS}{\mathbb{S}}
\newcommand{\bN}{\mathbb{N}}
\newcommand{\bZ}{\mathbb{Z}}
\newcommand{\bT}{\mathbb{T}}
\newcommand{\bC}{\mathbb{C}}

\newcommand{\bD}{\mathbb{D}}
\newcommand{\bH}{\mathbb{H}}

\newcommand{\bP}{\mathbb{P}}

\newcommand{\mbU}{\mathbf{U}}
\newcommand{\mbo}{\mathbf{0}}

\newcommand{\mbk}{\mathbf{k}}
\newcommand{\mbt}{\mathbf{t}}

\newtheorem{theorem}{\bf  Theorem}[section]
\newtheorem{lemma}{\bf  Lemma}[section]
\newtheorem{proposition}{\bf \sc Proposition}[section]
\newtheorem{corollary}{\bf \sc Corollary}[section]
\newtheorem{definition}{ \sc definition}[section]
\newtheorem{remark}{ \sc Remark}[section]



\begin{document}
	\begin{center}
		{\bf On multivariable matrix spectral factorization method}\\[5mm]
		Lasha Ephremidze${}^{a,b,*}$\,,\; Ilya M. Spitkovsky${}^a$
\vskip+0.2cm		
{\small			${}^a$ Division of Science and Mathematics, New York University Abu Dhabi (NYUAD), Saadiyat Island,  129188, Abu Dhabi, United Arab Emirates

${}^b$ Razmadze Mathematical Institute of I. Javakhishvili Tbilisi State University, 6. Tamarashvili Str., 0177, Tbilisi, Georgia}
	\end{center}
\let\thefootnote\relax\footnotetext{$*$ Corresponding author.\newline		
.{}$\;\;\;\;\;$ {\em	E-mail addresses:} le23@nyu.edu (L. Ephremidze), ims2@nyu.edu (I. M.  Spitkovsky). }
	\vskip+0.5cm
	{\small{\bf Abstract.} Spectral factorization is a prominent tool with several important applications in various areas of applied science. Wiener and Masani proved the existence of matrix spectral factorization. Their theorem has been extended to the multivariable case by Helson and Lowdenslager. Solving the problem numerically is challenging in both situations, and also important due to its practical applications. Therefore, several authors have developed algorithms for factorization. The Janashia-Lagvilava algorithm is a relatively new method for matrix spectral factorization which has proved to be useful in several applications. In this paper, we extend this method to the multivariable case. Consequently, a new numerical algorithm for multivariable matrix spectral factorization is constructed.

		\vskip+0.6cm
		
	
		\vskip+0.6cm
		

		\section{introduction}

	Spectral factorization was initiated in the works of Wiener \cite{Wiener49} and Kolmogorov \cite{Kolm41} as the scalar spectral factorization problem in relation to linear prediction theory of stationary stochastic processes, it has been extended to the matrix case by Wiener and Masani \cite{Wie57}. Their matrix spectral factorization (MSF) theorem asserts that if $S$ is a positive definite integrable $d\times d$ matrix function defined on the unit circle $\bT$ in the complex plane, $S\in L^1(\bT)^{d\times d}$, which satisfies the
		Paley-Wiener condition
		\begin{equation}\label{PW}
		\log \det S\in L^1({\mathbb T}),
		\end{equation}  then it admits the factorization
		\begin{equation}\label{MSF}
		S(t)=S_+(t)S_+^*(t).
		\end{equation}
Here  $S_+\in \bH^2(\bT)^{d\times d}$, i.e., $S_+$ can be analytically extended  inside $\bT$ to a square integrable matrix function (for exact definitions see Sect. 2) and $A^*$ stands for the Hermitian conjugate of $A$. The spectral factor $S_+$ can be selected outer and it is the unique up to a constant right unitary factor.
		
		Representation \eqref{MSF} plays a crucial role in the study of systems of singular integral equations \cite{GohKrein},  in linear estimation \cite{Kai99}, quadratic and $H^\infty$ control \cite{AndMoo}, \cite{Francis},  communications \cite{fisher},  wavelets and filter design \cite{Dau}, \cite{StrNgu}, Granger causality estimation in neuroscience \cite{Dhamala}, etc. In many of these applications, it is important to actually compute $S_+$ approximately for a given matrix function $S$ which becomes a challenging problem. Therefore, starting with Wiener's original efforts \cite{Wie58} to create a sound computational method of MSF, dozens of different algorithms have appeared in the literature (see the survey papers  \cite{Kuc}, \cite{SayKai} and  references therein, and also \cite{Bott13}, \cite{Jaf} for more recent results). 
		
A novel approach to the approximate factorization problem  \eqref{MSF}, without imposing any  restriction on $S$ beyond the necessary and sufficient condition \eqref{PW} for the existence of spectral factorization, was originally developed by Janashia and Lagvilava in \cite{JL99} for $2\times 2$ matrices. This approach was subsequently  extended to matrices of arbitrary dimension in \cite{IEEE}, efficiently  algorithmized  in \cite{IEEE-2018}, and successfully applied, e.g., in \cite{CNR}.

		Helson  and Lowdenslager \cite{HelLow58} further generalized Wiener-Masani MSF theorem to the multivariable case. To this end, let $N>1$ be a positive integer and let $H_N\sbs\bZ^N$  be the {\em half-plane of lattice points} defined recursively:
		$H_1=\bZ_+=\bN\cup\{0\} $ and 	$H_N=\{(k_1,k_2,\ldots,k_N)\in\bZ^N: k_1>0 \text{ or }
		k_1=0 \text{ and } (k_2,\ldots,k_N)\in H_{N-1}\}$.		
		We say that $f\in L^1(\bT^N)$ is of {\em analytic type} (with respect to the half-plane $H_N$) if $C_\mbk\{f\}=0$ for each $\mbk\in \bZ^N\setminus H_N$, where 
		 $C_{\mbk}\{f\}$ are the Fourier coefficients of $f$.		
		 The set of such functions will be denoted by $\calA(\bT^N)$. Finally, let  $\bH^2(\bT^N)=\calA(\bT^N)\cap L^2(\bT^N)$.
		 
		 The Helson-Lowdenslager MSF theorem \cite{HelLow58} asserts that if   \begin{equation}\label{S}
		 0<S\in L^1(\bT^N)^{d\times d}		 
		 \end{equation} 
		  and satisfies the condition
		 \begin{equation}\label{HL-PW}
		 \log \det S\in L^1({\mathbb T^N}),
		 \end{equation} 
		  then there exists a unique (up to a constant unitary matrix)  factorization
		 \begin{equation}\label{HL-MSF}
		 S(\mbt)=S_+(\mbt)S_+^*(\mbt),\;\;\;\mbt\in\bT^N,
		 \end{equation}
		 where  $S_+\in \bH^2(\bT^N)^{d\times d}$ is a matrix function of outer analytic type (see Sect. 2 for definitions).
		 
		 Wiener-Masani MSF theorem is used to process vector data depending on a single parameter, e.g., stationary time series collected by simultaneous observations at several different locations. However, due to the complex nature of the phenomena, data might be dependent on several parameters, e.g., color images on 2-D screen, or 3-D tomographic medical images. In such situations, the  Helson-Lowdenslager MSF theorem enters the  scene. Therefore  a lot of effort was put in the development of computational methods for $N$-D MSF  \cite{2-D-Wil}, \cite{Kummert}, \cite{N-D-1}, \cite{Basu}  \cite{Lai-2006}, \cite{Trentelman}, \cite{Annals-2004}.	 
		Clearly, improved methods of such factorization will further increase the  applicability of the Helson-Lowdenslager theorem.
		 
		 In this paper we extend the Janashia-Lagvilava method of MSF to the multivariable case, and hence introduce a novel computational algorithm for  the Helson-Lowdenslager matrix spectral factorization.  
		 The paper is organized as follows. In Section 2, we introduce necessary notation and preliminary observations. In Section 3, we consider the uniqueness of $N$-D MSF. Section 4 deals with multivariable scalar spectral factorization.  In Sections 5 and 6, we give an  essential component of the proposed multivariable MSF algorithm and present its general description.  We prove the convergence properties of the method in  Section 7 and  provide some results of numerical simulations in  Section 8. Finally, in the Appendix, we demonstrate the application of spectral factorization in Granger causality.

		 \section{Notation and preliminary observations}

		 Throughout the paper, a positive integer $N\geq 1$ denotes the dimension of the torus $\bT^N$. The latter  is equipped with the normalized Lebesgue measure $\mu_N= d\mbt/(2\pi)^N$. The half-plane of lattice points $H_N$ is defined in  the Introduction.  Note that $H_N$ has the following properties: i) $\mbo\in H_N$; ii) $\mbk\in H_N$ if and only if $-\mbk\not\in H_N$ unless $\mbk=\mbo$;  iii) $\mbk_1, \mbk_2\in H_N$ imply $\mbk_1+\mbk_2\in H_N$.

The complex conjugate of $a\in\bC$ is denoted by $\ol{a}$  and  $A^*$ stands for the Hermitian conjugate of $A\in \bC^{d\times d}$.
 For any set $\bS$, the notation $\bS^{d\times d}$ is used for the set of  ${d\times d}$ matrices with entries from $\bS$. For $M\in \bS^{d\times d}$ and $m\leq d$, $[M]_{m\tm m}$ denotes the $m\tm m$ leading principle submatrix of $M$.
A matrix function $S$ is called {\em factorable} if \eqref{S} and \eqref{HL-PW} hold. The notation $S>0$ means that it is positive definite a.e.

		 Let $L^p(\bT^N)$, $p>0$, be the standard Lebesgue space of $p$-integrable functions with usual definition of the norm $\|f\|_{L^p(\bT^N)}$ for $p\geq 1$.

		 The Fourier coefficients of $f\in L^1(\bT^N)$ are defined by the formula
		$$
		 C_{\mbk}\{f\}=\int_{\bT^N}f(\mbt)\, \mbt^{-\mbk}\,d\mu_N,
		$$
		 where $ \mbt^\mbk=t_1^{k_1}t_2^{k_2}\ldots t_N^{k_N}$ for 
		 $\mbt=(t_1,t_2,\ldots,t_N)\in\bT^N$, 
		 $\mbk=(k_1,k_2,\ldots,k_N)$ $\in\bZ^N$,		 
		  and $f\in \calA(\bT^N)$ means that $C_\mbk\{f\}=0$ for each multi-index $\mbk$ outside $H_N$ (as in the Introduction),
$$
\calA(\bT^N):=\{f\in L^1(\bT^N): C_{\mbk}\{f\}=0 \text{ for each } \mbk\notin H_N\}.
$$

On several occasions,  we need to expand a function $f\in L^2(\bT^N)$ into ``Fourier” series with respect to the first variable
\begin{equation}\label{20.8.1}
f(t_1,t_2,\ldots,t_N)=\sum\nolimits_{k\in\bZ} t_1^k C_{1k}\{f\}(t_2,\ldots,t_N)\;\text{ where } C_{1k}\{f\}\in L^2(\bT^{N-1}).
\end{equation}
For each $k$, the function $C_{1k}\{f\}$ is defined a.e. on $\bT^{N-1}$ by
$$
C_{1k}\{f\}(t_2,\ldots,t_N)=\frac{1}{2\pi}\int_\bT f(t_1,t_2,\ldots,t_N)t_1^{-k}\,dt_1,
$$
and equation \eqref{20.8.1} holds for a.e. $(t_1,t_2,\ldots,t_N)\in\bT^N$.

 If a function $f\in L^2(\bT^N)$ has the form
$$
f(t_1,t_2,\ldots,t_N)=\sum\nolimits_{k=0}^n t_1^k \al_{k}(t_2,\ldots,t_N)\;\text{ where } \al_k\in L^2(\bT^{N-1}),
$$
then we say that
\begin{equation}\label{06.04.01}
f\in\calP_+^n(\bT^N_1).
\end{equation}

For a function $f$ defined by \eqref{20.8.1}, we let
\begin{equation}\label{20.8.2}
\wdt{f}(t_1,t_2,\ldots,t_N)=\sum\nolimits_{k\in\bZ} t_1^{-k} \ol{C_{1k}\{f\}(t_2,\ldots,t_N)}.
\end{equation}
Clearly
$$
\wdt{f}(t_1,t_2,\ldots,t_N)=  \ol{f(t_1,t_2,\ldots,t_N)}\;\text{ a.e. on } \bT^N.
$$

 For a measurable function $f:\bT^N\to\bC$ we define $	f_{t_2,t_3,\ldots,t_N}:\bT\to\bC$ for a.a. $(t_2,t_3,\ldots,t_N)\in \bT^{N-1}$ by 
 $$
 f_{t_2,t_3,\ldots,t_N}(t)=f(t,t_2,t_3,\ldots,t_N).
 $$
 Obviously, because of Fubini’s theorem,
 \begin{equation}\label{2.51}
 f\in L^p(\bT^N)
\Longrightarrow
 f_{t_2,t_3,\ldots,t_N}\in L^p(\bT)\;\text{ for a.e. }
 (t_2,t_3,\ldots,t_N)\in \bT^{N-1}.
 \end{equation}
 For $ f\in L^1(T^N)$, let  $\hat{f}\in L^1(\bT^{N-1})$ be defined by
 \begin{equation}\label{2.81}
 \hat{f}(t_2,t_3,\ldots,t_N)=  \int_{\bT}  f_{t_2,t_3,\ldots,t_N}(t) \,d\mu_1
=  \int_{\bT}f(t,t_2,t_3,\ldots,t_N)\,d\mu_1.
 \end{equation}	
 
  Let 
 \begin{equation}\label{defHpN}
 \bH^p(\bT^N):=\calA(\bT^N)\cap L^p(\bT^N),\;\text{ where }\;p\geq 1, 
 \end{equation}
 be the class of analytic type functions (defined in  the Introduction for $p=2$).  
 The following recursive characterization of $\bH^p(\bT^N)$ will be useful in the sequel.
		 
		 	\begin{proposition}\label{H1}
Let $f\in L^p(\bT^N)$, where   $p\geq 1$ and $N\geq 2$. Then 
		 	\begin{equation}\label{HpTN}
		 	f\in\bH^p(\bT^N)
		 	\end{equation}
	if and only if
		 	\begin{equation}\label{Hp}
		 	f_{t_2,t_3,\ldots,t_N}\in \bH^p(\bT)
		 	\end{equation}
	for a.e. $(t_2,t_3,\ldots,t_N)\in\bT^{N-1}$ and 
		 	 \begin{equation}\label{HpT}
		 	\hat{f}\in \bH^p(\bT^{N-1}).
		 	 \end{equation}
 \end{proposition}
	 
		 \begin{proof} 
We have  $f_{t_2,t_3,\ldots,t_N}\in L^p(\bT)$ for a.e. 
$(t_2,t_3,\ldots,t_N)\in \bT^{N-1}$	due to \eqref{2.51}, and
it follows from definition \eqref{2.81}, Jensen's inequality $\big(\int_{\bT}|g|\,d\mu_1\big)^p\leq\int_{\bT}|g|^p\,d\mu_1$, and 
Fubini’s theorem that
 $\hat{f}\in L^p(\bT^{N-1})$.
		 	
	Suppose \eqref{HpTN} holds. For each integer $k<0$, the function $h_k\in L^1(\bT^{N-1})$ defined by 
		 	\begin{equation}\label{h}
		 	h_k(t_2,t_3,\ldots,t_N)=\int_{\bT}f(t,t_2,t_3,\ldots,t_N)\,t^{-k}\,d\mu_1
		 	\end{equation}
		 	has all Fourier coefficients equal to $0$, because
		 	$$
		 	C_{\mbk}\{h_k\}=\int_{\bT^N}f(t,\mbt)\,t^{-k}\mbt^{-\mbk}\,d\mu_1d\mu_{N-1}=0
		 	$$
for any $\mbk=(k_2,k_3,\ldots,k_N)$, where	$\mbt=(t_2,t_3,\ldots,t_N)$, since $f\in\calA(\bT^N)$. Hence, for a.e. $(t_2,t_3,\ldots,t_N)\in\bT^{N-1}$, the integral in \eqref{h} is equal to $0$ for all $k<0$, and therefore \eqref{Hp} holds. For $k=0$ and $\mbk\notin H_{N-1}$, we still have

		 	$$
		 	0=\int_{\bT^N}f(t,\mbt)\,t^{-k}\mbt^{-\mbk}\,d\mu_1d\mu_{N-1}=
    \int_{\bT^{N-1}}\hat{f}(\mbt)\mbt^{-\mbk}\,d\mu_{N-1},
		 	$$
		 	and therefore $\hat{f}\in\calA(\bT^{N-1})$ and \eqref{HpT} holds.
		 	
Suppose now that \eqref{Hp} and \eqref{HpT} hold. Then $f\in\calA(\bT^N)$ can be proved by direct application of Fubini's theorem reversing the above obtained implications. Therefore \eqref{HpTN} holds. 
\end{proof} 

Next, for convenience of presentation of the obtained results, we introduce the Hardy spaces
	\begin{equation}\label{2.17}
\bH^p(\bT^N)\; \text{ for }\; p>0.
\end{equation}
For  $N=1$, the Hardy space $\bH^p=\bH^p(\bT)$ is defined for all $p>0$ by
		 	$$
		 	\bH^p:=\left\{f\in\mathcal{A}(\mathbb{D}):\sup\limits_{\rho<1}
		 	\int\nolimits_0^{2\pi}|f(\rho e^{i\theta})|^p\,d\theta<\infty\right\}.
		 	$$
The functions from $\bH^p$, where $p>0$, and their boundary values can be identified (see, e.g. \cite{Koo}). Therefore, we can assume that $\bH^p=\bH^p(\bT)\sbs L^p(\bT)$ for $0<p\leq\infty$, and this definition  agrees with \eqref{defHpN} for $p\geq 1$ and $N=1$. (However, we can speak about the values of a function $f\in \bH^p(\bT)$ inside the unit disk if necessary). Nevertheless, the definition \eqref{defHpN} cannot be extended to arbitrary $p>0$ because the question whether $f\in \calA(\bT^N)$ arises only when $f$ is integrable. However, the equivalent characterization of $\bH^p(\bT^N)$ according to Propositiom 2.1 enables us to extend this definitions to \eqref{2.17}.

\begin{definition}\label{def1}
Assume $\bH^p(\bT^1)=\bH^p(\bT)=\bH^p$. 
We say that $f\in\bH^p(\bT^N)$, where $p>0$ and $N\geq 2$, if and only  if $	f_{t_2,t_3,\ldots,t_N}\in \bH^p$ for a.a. $(t_2,t_3,\ldots,t_N)\in\bT^{N-1}$   and $\hat{f}_1\in \bH^p(\bT^{N-1})$, where $\hat{f}_1$ is defined by the equality 
	\begin{equation}\label{def2.12}
	 		\hat{f}_1(t_2,t_3,\ldots,t_N):=f_{t_2,t_3,\ldots,t_N}(0)=f_{t_2,t_3,\ldots,t_N}(z)|_{z=0}\,.
	\end{equation}
\end{definition}

 Note that this definition of $\bH^p(\bT^N)$   differs from the standard Hardy space defined in the theory of several complex variables for $N>1$ (see \cite{Rudin1980}, p. 84).

 \begin{remark}\label{remDef2.1}
 	It follows from Definition \ref{def1} that if $f\in \bH^p(\bT^N)$, then 
 	$f(z,t_{2},\ldots,t_N)$ is defined for a.a. 
 	$(t_2,\ldots,t_N)\in\bT^{N-1}$ and
 	each $z\in\bD$.  	
Furthermore, for  $1\leq k<N$,  	
 $f(0,\ldots,0,z,t_{k+1},\ldots,t_N)$
 is defined for a.a. 
 $(t_k,t_{k+1},\ldots,t_N)\in\bT^{N-k}$ and
 each $z\in\bD$.	
 	 Therefore, similarly to \eqref{def2.12}, one can define the function $\hat{f}_k:\bT^{N-k}\to\bC$ by
 	\begin{equation}\label{def2.13}
 	\hat{f}_k(t_{k+1}\ldots,t_N):=f(0,\ldots,0,0,t_{k+1},\ldots,t_N)
 	\end{equation}	

	If $p\geq 1$, then $\hat{f}_1=\hat{f}$ a.e. on $\bT^{N-1}$, where $\hat{f}$ is defined by \eqref{2.81}, and
\begin{equation}\label{def2.14}
\hat{f}_k(t_{k+1}\ldots,t_N)=\int_{\bT^k}f(\cdot,t_{k+1},\ldots,t_N)\,d\mu_k.
\end{equation}
In particular, 
	\begin{equation}\label{foLnk}
f(\mbo):=f(0,0,\cdots,0,0)=\int_{\bT^N} f\,d\mu_N=C_\mbo\{f\}
\end{equation}
$($the definition in \eqref{foLnk} makes sense, and thus will be used, for all $p>0${}$)$.	
	\end{remark}
	 	
The prominent property of Hardy space functions
$$
\int_{\bT} \log|h(t)|\,d\mu_1>-\iy,
$$ 		
for any  $0\not\equiv h\in\bH^p$ (see \cite[Th. 17.17]{Rud}), is no longer valid for arbitrary function $0\not\equiv f\in\bH^p(\bT^N)$ for $N>1$, because it may happen that 
	\begin{equation}\label{eq4}
\int_{\bT^N} \log |f(\mbt)|\,d\mu_N=-\iy
\end{equation}
(see a counterexample at \cite[p. 176]{HelLow58}). However, it is possible to single out the situations where  \eqref{eq4} may occur.

\begin{lemma}\label{Lmnk}
	Let $f\in\bH^P(\bT^N)$, for $p>0$ and $N\geq 2$, and suppose \eqref{eq4} holds. Then 
	\begin{equation}\label{eqv0}
\hat{f}_{N-1}\equiv 0.
\end{equation}	
\end{lemma}

\begin{proof}
We use the well-known estimation 
\begin{equation}\label{logh}
\log|h(0)|\leq\int_{\bT} \log |h(t)|\,d\mu_1
\end{equation}
for any $h\in\bH^p$ (see \cite[Th. 17.17]{Rud}), which together with \eqref{def2.12} implies that 
\begin{equation}\label{gath16}
\int_{\bT^N}\log|f(\mbt)|\,d\mu_N
\geq \int_{\bT^{N-1}} \log|\hat{f}_1(t_2,t_3,\ldots,t_N)|\,d\mu_{N-1}.
\end{equation}
Hence, it follows from \eqref{eq4} that the second integral in \eqref{gath16} is also $-\infty$.

We can carry out the same reasoning for the function $\hat{f}_1$ instead of $f$, and continuing recursively  in the same manner, we will obtain 
$$ \int_{\bT^{N-k}} \log|\hat{f}_k|\,d\mu_{N-k}=-\iy
\;\text{   for }\; k=1,2,\ldots,N-1.
$$
Hence, we get $\hat{f}_{N-1}\in \bH^p(\bT)$ and $\int_{\bT} \log|\hat{f}_{N-1}|\,d\mu_1=-\iy$, which implies \eqref{eqv0}.
\end{proof}

Next we define the set of {\em outer} type functions from $\bH^p(\bT^N)$, $p>0$, which is denoted by $\bH^p_O(\bT^N)$. For $N=n=1$, the definition is classical:  $0\not\equiv f\in\bH^p$, where $p>0$, is called outer, $f\in\bH^p_O(\bT)=:\bH^p_O$, if  
	\begin{equation}\label{outer}
	 		f(z)=c\cdot  \exp\left(\frac 1{2\pi}
	 		\int\nolimits_0^{2\pi}\frac{e^{i\theta}+z}{e^{i\theta}-z}\log
	 		\big|f(e^{i\theta})\big|\,d\theta\right),\;\;\;\;\;|c|=1,
	\end{equation}
which is equivalent to	(see \cite[Th. 17.17]{Rud})	
		\begin{equation}\label{outer2}
\log|f(0)|=\frac 1{2\pi} 	\int\nolimits_0^{2\pi}\log
	\big|f(e^{i\theta})\big|\,d\theta.
	\end{equation}
For $N>1$, the definition will be given recursively.
\begin{definition}\label{def2}
	 We say that $f\in\bH^p_O(\bT^N)$, $p>0$, if and only  if 
	\begin{equation}\label{outae}
	f_{t_2,t_3,\ldots,t_N}\in \bH^p_O\; \text{ for a.e. } (t_2,t_3,\ldots,t_N)\in\bT^{N-1} 
	\end{equation} 
 and $\hat{f}_1\in \bH^p_O(\bT^{N-1})$, where $\hat{f}_1$ is defined by \eqref{def2.12}.
\end{definition}

Note that the  definition of $\bH^p_O(\bT^N)$ coincides with the corresponding concept introduced in \cite[p. 181]{HelLow58} for $p\geq 1$. Namely, the following lemma holds.
\begin{lemma}\label{Louter}
	Let $f\in\bH^p(\bT^N)$, $p>0$. Then 
	\begin{equation}\label{2.195}
f\in\bH^p_O(\bT^N)
 \end{equation}
 if and only if 
 	\begin{equation}\label{logeq}
\int_{\bT^N}\log|f(\mbt)|\,d\mu_N=\log\left|f(\mbo)\right|>-\iy
 \end{equation}
 where $f(\mbo)$ is defined by \eqref{foLnk}.
\end{lemma}
\begin{proof}
As mentioned above, the equivalence of these two conditions is a well-known fact for $N=1$, therefore, we need to consider the case $N\geq 2$.

 Note first that $f\in\bH^p_O(\bT^N)$ $\Longrightarrow$ $\int_{\bT^N}\log|f(t)|\,d\mu_N\not=-\iy$ since otherwise $\hat{f}_{N-1}\equiv 0$ by  Lemma \ref{Lmnk}, which contradicts the definition of $\bH^p_O(\bT^N)$ space.
 
 On the other hand, both \eqref{2.195} and \eqref {logeq} are equivalent to the sequence of equations
 \begin{equation}\label{2.234}
  \int_{\bT^N}\log|f|\,d\mu_N=
 \int_{\bT^{N-1}}\log|\hat{f}_1|\,d\mu_{N-1}=
  \ldots=\int_{\bT}\log|\hat{f}_{N-1}|\,d\mu_1=
 \log\left|f(\mbo)\right|
 \end{equation}
(because of successive application of \eqref {gath16}, where we have to have ``$=$” instead of ``$\geq$”; see also \eqref{foLnk}).   Therefore, \eqref{2.195} and \eqref {logeq} are equivalent.
\end{proof}
	
	\begin{definition}\label{def3}
		We say that a matrix function $F\in H^p(\bT^N)^{d\times d}$ is of outer type and use the notation $F\in H^p(\bT^N)_O^{d\times d}$ if 
		$\det F\in \bH^{r}_O(\bT^N)$ for some $r>0$.
	\end{definition}	 	

Next, we introduce the following imbedded spaces, $\bH^p(\bT^N)\sbs \bH^p(\bT^N_{N-1})\sbs
\ldots\sbs\bH^p(\bT^N_1)\sbs  L^p(\bT^N)$, which is used later.
 \begin{definition}\label{def2.4}
 	For $N\geq 2$ and $1\leq l<N$, we say that $f\in \bH^p(\bT^N_l)$ $($resp. $f\in \bH^p_O(\bT^N_l)${}$)$ if and only if $f\in L^p(\bT^N)$ and, for a.a. $(t_{l+1},t_{l+2},\ldots,t_N)\in\bT^{N-l}$,
 $$
 f(\cdot,t_{l+1},\ldots,t_N)\in\bH^p(\bT^l)\;\;(\text{resp. }  f(\cdot,t_{l+1},\ldots,t_N)\in\bH^p_O(\bT^l) )
 $$
 as a function of variables $(t_1,t_2,\ldots,t_l)$.
 
 The classes of matrix functions $ \bH^p(\bT^N_l)^{d\times d}$ and $ \bH^p(\bT^N_l)^{d\times d}_O$ are defined similarly.

 \end{definition}

\begin{remark}\label{rem2.3}
	Note that if $f\in H^p(\bT^N_l)$ and $k\leq l$, then Remark \ref{remDef2.1} remains valid and $\hat{f}_k$ can be defined by \eqref{def2.13}.  Furthermore, if
	$f,g\in H^p_O(\bT^N_l)$ and $|f|=|g|$ a.e. on $\bT^N$, then 
	$$
	|\hat{f}_{k-1}(z,t_{k+1},t_{k+2},\ldots,t_N)|= |\hat{g}_{k-1}(z,t_{k+1},t_{k+2},\ldots,t_N)| 
	$$
 for a.a.  $(t_{k+1},t_{k+2},\ldots,t_N)\in\bT^{N-k}$  and each $z\in\bD$ $($it is assumed that $\hat{f}_0=f${}$)$, consequently,
 $$
 |\hat{f}_k|=|\hat{g}_k|\; \text{ a.e. on }\bT^{N-k}
 $$
 for each $k=1,2,\ldots,l$.
 \end{remark}

\begin{remark}\label{rem2.4}
	It follows from Definitions \ref{def1},  \ref{def2}, and \ref{def2.4} that if $f\in \bH^p_O(\bT^N_l)$, where $l<N$, and
$$
g(t_1,t_2\ldots,t_N)=f(t_1,t_2\ldots,t_N)h(t_{l+1},t_{l+2}\ldots,t_N)
$$
for some $h\in L^\iy(\bT^{N-l})$, then $g\in \bH^p_O(\bT^N_l)$ as well. This fact is often tacitly used in what follows.
\end{remark}

 A matrix function $U\in L^\iy(\bT^N)^{d\times d}$ is called unitary if $U(\mbt)U^*(\mbt)=I_d$ for a.e. $\mbt\in\bT^N$, where $I_d$ stands for the $d\times d$ unit matrix.

Integration of matrix functions and convergence of matrix valued sequences are understood entry-wise. 

In Section 7 we use the following stability result on matrix spectral factorization proved in \cite{EJL11} for $N=1$.

\begin{theorem}\label{th.2.1}
	$($\cite{EJL11},{\rm Th. 1}$)$ 
	Let $0<S^{\{n\}}\in L^1(\bT)$, $n=0,1,2,\ldots$, be a sequence of positive definite integrable matrix functions such that
$$
	\|S^{\{n\}}-S^{\{0\}}\|_{L^1(\bT)}\to 0 \text{ and }\; 
	\int_{\bT}\log\det S^{\{n\}}(t)\,dt\to  	\int_{\bT}\log\det S^{\{0\}}(t)\,dt.
$$
Then
$$
\|S^{\{n\}}_+-S^{\{0\}}_+\|_{L^2(\bT)}\to 0.
$$
\end{theorem}

The proof of the following proposition, which is valid for arbitrary finite measure space,  follows  easily  from
the necessary and sufficient condition for the convergence in the
norm: $\|f_n-f\|_{L^p}\to 0$, where $p\geq 1$, if and only if $f_n\rightrightarrows
f$ and $\sup_{n>k,\mu(E)<\delta}\int_E|f_n|^pd\mu \to 0$ as $k\to\infty,
\delta\to 0$, where $\rightrightarrows$ stands for the convergence in measure.

\begin{proposition}
	If $h_n\in L^1$, $f_n\in L^p$, $p\geq 1$, $ n=0,1,\dots$, $ f_n \rightrightarrows f_0$, $	|f_n(t)|^p\leq |h_n(t)|$ and $\|h_n-h_0\|_{L^1}\to 0$, then
	$\|f_n-f_0\|_{L^p}\to 0$.
\end{proposition}

\begin{corollary}
	If $\|f_n\!-\!f\|_{L^p}\to 0$, $p\!\geq\!1$, and
	$|u_n(t)|\!\leq\!1$, $n\!=\!0,1, \dots, u_n\rightrightarrows u$,
	then $\|f_nu_n-fu\|_{L^p}\to 0$.
\end{corollary}

		\section{Uniqueness of multivariable matrix spectral factorization}

Definitions \ref{def1} and \ref{def2} allow us to formulate generalized Smirnov's  theorem (namely, $L^p(\bT)\ni f=g/h$, $g\in \bH^q$, $h\in\bH^r_O$ $\implies$ $f\in\bH^p$; see \cite[p.109]{Koo}) for the $N$-dimensional case. This result is used to provide a simple proof of the uniqueness of the Helson-Lowdenslager MSF theorem (see Proposition 3.2 below). Note that the uniqueness is not discussed  in the original formulation of this theorem in \cite{HelLow58}, \cite{HelLow61}.

\begin{proposition}\label{Smirnov2}
	Suppose $f\in L^p(\bT^N)$ can be represented as a ratio 
	\begin{equation}\label{ratio}
	f=\frac{g}{h}
	\end{equation}
	with $g\in\bH^q(\bT^N)$ and $h\in\bH^r_O(\bT^N)$, where $p,q,r>0$ are arbitrary. Then
	\begin{equation}\label{fHpTN}
	f\in\bH^p(\bT^N).
	\end{equation}
\end{proposition} 

First we prove the following 
\begin{lemma}\label{fLp}
	If $f\in L^p(\bT^N)$, $p>0$, and 
	$f_{t_2,t_3,\ldots,t_N}\in\bH^p$ 	for a.e. $(t_2,t_3,\ldots,t_N)\in\bT^{N-1}$, then $($see \eqref{def2.12}$)$
	\begin{equation}\label{f0Lp}
	\hat{f}_1\in L^p(\bT^{N-1}).
	\end{equation}
\end{lemma}
\begin{proof}
	Indeed, if $h\in\bH^p$, $p>0$, then $|h|^p=\exp(p\log|h|)$ is a subharmonic function in $\bD$, and therefore $	\int\nolimits_0^{2\pi}|h(\rho e^{i\theta})|^p\,d\theta$ is increasing on $(0,1)$ as a function of $\rho$ (see \cite[\S 1.6]{Garn}). Thus, $|h(0)|^p=\lim_{\rho\to 0+}(1/2\pi)\int\nolimits_0^{2\pi}|h(\rho e^{i\theta})|^p\,d\theta\leq \lim_{\rho\to 1-}(1/2\pi)\int\nolimits_0^{2\pi}|h(\rho e^{i\theta})|^p\,d\theta =\int_{\bT} |h|^p\,d\mu_1.$ 
	Consequently
$$
	\int_{\bT^{N-1}}|\hat{f}_1|^p\,d\mu_{N-1}\!=\!\!\int_{\bT^{N-1}}|f(0,\cdot)|^p\,d\mu_{N-1}\leq \int_{\bT^{N-1}}\left(\int_{\bT}|f(t,\cdot)|^p\,d\mu_1\right)\,d\mu_{N-1}\!=\!\!
	\int_{\bT^N}|f|^p\,d\mu_N.
$$
	Thus, \eqref{f0Lp} holds.
\end{proof}

{\em Proof of Proposition \ref{Smirnov2}.} The proof can be carried out by induction with respect to $N$ and the goal is achieved by using Definitions \ref{def1} and \ref{def2}. Indeed, for $N=1$, the statement amounts to the above-mentioned generalized Smirnov's theorem. Thus, we can make the assumption that the proposition is correct if we take $N-1$ instead of $N$. 

On the other hand, the hypothesis of the proposition implies that $$f_{t_2,t_3,\ldots,t_N}=g_{t_2,t_3,\ldots,t_N}/h_{t_2,t_3,\ldots,t_N}\in \bH^p$$ due to the one-dimensional theorem, as long as \eqref{2.51} holds and Definitions 2.1 and 2.2 imply that $g_{t_2,t_3,\ldots,t_N}\in \bH^q$ and $h_{t_2,t_3,\ldots,t_N}\in \bH^r_O$. This in turn implies \eqref{f0Lp} by Lemma \ref{fLp}. We also have 	$\hat{f}_1=f(0,\cdot)=g(0,\cdot) /h(0,\cdot)=\hat{g}_1/\hat{h}_1$ with $\hat{g}_1\in\bH^p(\bT^{N-1})$ and $\hat{h}_1\in\bH^p_O(\bT^{N-1})$ (by virtue of Definitions \ref{def1} and \ref{def2}). Hence, $\hat{f}_1\in \bH^p(\bT^{N-1})$ by the assumption of the induction and \eqref{fHpTN} holds by  Definition \ref{def1}. \hfill$\Box$

\begin{remark}\label{rem3.1}
	Using H\"{o}lder's inequality and  induction  similar to the proof above, one can prove that if  $f\in \bH^p(\bT^N)$ and $g\in \bH^q(\bT^N)$, then $fg\in \bH^{\frac{pq}{p+q}}(\bT^N)$. $($Therefore, the exponent $r$ in Definition \ref{def3}  can be taken equal to $p/d$.$)$ Furthermore, if $f\in \bH^p_O(\bT^N)$ and $g\in \bH^q_O(\bT^N)$ then $fg\in \bH^{\frac{pq}{p+q}}_O(\bT^N)$.
\end{remark}

We are now ready to prove the uniqueness of factorization in the Helson-Lowdenslager MSF theorem, which is similar to the one presented in \cite{EJL09} for the Wiener-Masani MSF theorem.

\begin{proposition}
	Let $0<S\in L^1(\bT^N)^{d\times d}$ and suppose \eqref{HL-PW} holds. If
	\begin{equation}\label{SS}
	S=S_+S_+^*=\Xi_+\Xi_+^*
	\end{equation}
	are two spectral factorizations of $S$ with spectral factors of outer type,
	$S_+$, $\Xi_+\in  \bH^2_O(\bT^N)^{d\times d}$, then there exists a constant unitary matrix $U\in\bC^{d\times d}$ such that
	\begin{equation}\label{SU}
	S_+=\Xi_+U.
	\end{equation}
\end{proposition}

\begin{proof}
	The equations in \eqref{SS} imply that 
	$$
	\big(\Xi_+^{-1}S_+\big)\big(S_+^*(\Xi_+^*)^{-1}\big)=I_d,
	$$
	so that $U=\Xi_+^{-1}S_+$ is a unitary matrix function,  i.e.,
$U^{-1}=U^*$  a.e. on $\bT^N$.
	The entries of a unitary matrix are bounded. Hence
	$$
	U\in L^\iy(\bT^N)^{d\times d}.
	$$
	Since $\det\Xi_+$ is of outer analytic type and $\Xi_+^{-1}=(\det\Xi_+)^{-1}\adj(\Xi_+)$, we have 
	$$
	U=\frac{V}{w}\,, \;\text{ where }\; V\in\bH^q (\bT^N)^{d\times d} \text{ and } w\in\bH^r(\bT^N) \text{ for some } q,r>0
	$$
	(see Remark \ref{rem3.1}). Hence we can apply Proposition \ref{Smirnov2} for the entries of $U$ and conclude that
$$
	U\in \bH^\iy(\bT^N)^{d\times d}.
$$
	By changing the roles of $\Xi_+$ and $S_+$ in this discussion, we get 
$$
	U^{-1}\in \bH^\iy(\bT^N)^{d\times d}.
$$
Hence, 
	$$
	U,\,U^{*}\in \bH^\iy(\bT^N)^{d\times d}.
	$$
which means that entries of $U$ have all Fourier coefficients except $C_\mbo$ equal to $0$. Consequently, they are constant.
\end{proof}	

		\section{Multivariable 	scalar spectral factorization}
	
	Helson-Lowdenslager spectral factorization theorem in the scalar case asserts that: {\em if $0< f\in L^1(\bT^N)$ and $\int_{\bT^N}\log f\,d\mu_N>-\iy$, then there exists a unique $($up to a constant factor of modulus 1$)$ function $f_+\in\bH^2_O(\bT^N)$ such that }
	\begin{equation}\label{ssf1}
	f=|f_+|^2  \text{ a.e. on } \bT^N.
	\end{equation}
	
	If $N=1$, then the function $f_+$ can be written explicitly: $f_+(t)=\lim_{r\to 1-}f_+(rt)$, where
	$$
	f_+(z)=c\cdot \exp\left(\frac 1{4\pi}
	\int\nolimits_0^{2\pi}\frac{e^{i\theta}+z}{e^{i\theta}-z}\log
	f(e^{i\theta})\,d\theta\right),\;\;|c|=1,\;\;|z|<1
	$$
	(cf.  \eqref{outer}).
	Note that $f_+$ can be also written as 
\begin{equation*}
f_+(t)=c\cdot \sqrt{f(t)}\exp\left(\frac12 i \calS\big(\log f\big)(t)\right),
\end{equation*} 
where $\calS(f)$ stands for the conjugate of $f\in L^1(\bT)$:
$$
\calS(f)(e^{i\tau})=\frac{1}{2\pi}(P) \int_0^{2\pi}f(e^{i\theta}) \cot\frac{\tau-\theta}{2}\,d\theta.
$$	
	
	 In the multivariable case, it is sufficient for our purposes to construct a factorization
$$
f=f_{+,1}\ol{f_{+,1}},
$$
where	$f_{+,1}\in\bH_O^2(\bT^N_1)$. 
Such factorization can be written in the explicit form 
\begin{equation}\label{21.8.4}
f_{+,1}(\mbt)=\sqrt{f(\mbt)}\exp\left(i\calS_1\big(\log\sqrt{f(\mbt)}\big) \right)
\end{equation}
if we introduce the singular operator $\calS_1:L^1(\bT^N)\to L^p(\bT^N)$, $p<1$, with respect to the first variable by the formula	
$$
\calS_1(h)(e^{i\tau},t_{2},\ldots,t_N)=\frac{1}{2\pi}(P) \int_0^{2\pi}h(e^{i\theta},t_{2},\ldots,t_N) \cot\frac{\tau-\theta}{2}\,d\theta.
$$	
Note that we can write $f_+$ in \eqref{ssf1} explicitly as
\begin{equation}\label{4.23}
	f_+(t_1,t_2,\ldots,t_N)= {f_{+,1}(t_1,t_2,\ldots,t_N)}\prod_{k=1}^{N-1}\exp\big(i\calS_1(\check{f_k})(t_{k+1},\ldots,t_N)\big),
\end{equation}
	where  $\check{f_k}:\bT^{N-k}\to \bR$, $k=1,2,\ldots,N-1$, are the functions defined by (cf. \eqref{def2.13}):
\begin{equation}\label{4.4}
\check{f_k}(t_{k+1},\ldots,t_N)=\int_{\bT^{k}}
\log \sqrt{f(\cdot,t_{k+1},\ldots,t_N)}\,d\mu_k
\end{equation} 
In particular, for $2\leq l\leq N$, we have $f=f_{+,l}\ol{f_{+,l}}$, where
$$
{f_{+,{l}}(t_1,t_2,\ldots,t_N)}:=
{f_{+,1}(t_1,t_2,\ldots,t_N)}\prod_{k=1}^{l-1}\exp\big(i\calS_1(\check{f_k})(t_{k+1},\ldots,t_N)\big)\in H^2_O(T^N_{l})
$$
and, because of \eqref{outer2} and Fubini's theorem (see also \eqref{2.234}), the expression \eqref{4.4} is equal to
$$
\frac12 \int_0^{2\pi}\log |f_{+,l}(0,\ldots,0,e^{i\theta},t_{l+1},\ldots,t_N)|\,d\theta=\frac12
\log|f_{+}(0,\ldots,0,0,t_{l+1},\ldots,t_N)|.
$$

\section{Description of $S = S_{+,1} S_{+,1}^*$ factorization algorithm}

In this section, we describe the algorithmic steps for the factorization
\begin{equation}\label{5.11}
S(t)=S_{+,1}(t)S_{+,1}^*(t),
\end{equation}
where 
\begin{equation}\label{6.4}
S_{+,1}\in\bH^2(\bT^N_1)_O^{d\times d},
\end{equation}
of a matrix \eqref{S} which satisfies \eqref {HL-PW}. The existence of such factorization follows from the corresponding 1-D Wiener-Masani theorem since Fubini’s theorem guarantees that, for a.e. $(t_2,\ldots,t_N)\in\bT^{N-1}$,
$$
S(\cdot,t_2,\ldots,t_N)\in L^1(\bT)\;\text{ and }\; \log\det S(\cdot,t_2,\ldots,t_N)\in L^1(\bT).
$$
Below we provide constructive procedures for factorization \eqref{5.11}. 
These procedures stem from the corresponding 1-D MSF algorithm proposed in \cite{IEEE}. The main idea which  demonstrates the possibility of such generalization is presented in our recent paper \cite{GMJ20}.

\smallskip

{\bf Procedure 1.} We perform lower-upper factorization
\begin{equation}\label{7.1}
S(\mbt)=M_1(\mbt)M_1^*(\mbt),
\end{equation} 
where 
\begin{equation}\label{7.16}
M_1(\mbt)=\begin{pmatrix}f_1(\mbt)&0&\cdots&0&0\\
\xi_{21}(\mbt)&f_2(\mbt)&\cdots&0&0\\
\vdots&\vdots&\vdots&\vdots&\vdots\\
\xi_{d-1,1}(\mbt)&\xi_{d-1,2}(\mbt)&\cdots&f_{d-1}(\mbt)&0\\
\xi_{d1}(\mbt)&\xi_{d2}(\mbt)&\cdots&\xi_{d,d-1}(\mbt)&f_d(\mbt)
\end{pmatrix}
\end{equation}
with $f_i\in \bH^2_O(\bT^N_1)$, $1\leq i\leq d$, and $\xi_{ij}\in L^2(\bT^N)$, $2\leq i\leq d$, $1\leq j<i$. As in the $N = 1$ case, \eqref{7.1} can be achieved by pointwise Cholesky factorization of $S(\mbt)$ and then applying formula \eqref{21.8.4} for the diagonal entries.
Note that $\det M_1\in \bH^{2/d}_O(\bT^N_1)$.

\smallskip

{\bf Procedure 2.} The factor $S_{+,1}$ is represented as 
\begin{equation}\label{7.17}
S_{+,1}(\mbt)=M_1(\mbt)\mbU_2(\mbt)\mbU_3(\mbt)\ldots \mbU_r(\mbt).
\end{equation} 
Each $\mbU_m$, $m=2,3,\ldots,d$, has the form
\begin{equation}\label{6.5.6}
\mbU_m(\mbt)=\begin{pmatrix}U_{m}(\mbt)&0\\0&I_{r-m}\end{pmatrix},
\end{equation}
where $U_m$ is a unitary matrix function with a special structure
\begin{equation}\label{67.6}
U_m(\mbt)=\begin{pmatrix}u_{11}(\mbt)&u_{12}(\mbt)&\cdots&u_{1,m-1}(\mbt)&u_{1m}(\mbt)\\
u_{21}(\mbt)&u_{22}(\mbt)&\cdots&u_{2,m-1}(\mbt)&u_{2m}(\mbt)\\
\vdots&\vdots&\vdots&\vdots&\vdots\\
u_{m-1,1}(\mbt)&u_{m-1,2}(\mbt)&\cdots&u_{m-1,m-1}(\mbt)&u_{m-1,m}(\mbt)\\[2mm]
\wdt{u_{m1}}(\mbt)&\wdt{u_{m2}}(\mbt)&\cdots&\wdt{u_{m,m-1}}(\mbt)&\wdt{u_{mm}}(\mbt)\\
\end{pmatrix},
\end{equation}
(see \eqref{20.8.2}) with $u_{ij}\in\bH^\iy(\bT^N_1)$, such that 
\begin{equation}\label{67.7}
\det U_m(\mbt)=1 \;\text{ for a e. } \mbt\in\bT^N
\end{equation}
and 
$$
[Q_m]_{m\tm m}:=[M_1\mbU_2\ldots\mbU_m]_{m\tm m}\in \big(\bH^2(\bT^N_1)
\big)_O^{m\tm m}.
$$

\smallskip

{\bf Procedure 3.} The unitary matrix functions \eqref{6.5.6} are constructed recursively. We assume that $\mbU_2, \mbU_3,\ldots, \mbU_{m-1}$ have already been constructed and obtain $U_m$ by the following steps:

{\bf Step 1.} Consider the matrix function $F$ of the form
\begin{equation}\label{7F}
F(\mbt)=\begin{pmatrix}1&0&0&\cdots&0&0\\
0&1&0&\cdots&0&0\\
0&0&1&\cdots&0&0\\
\vdots&\vdots&\vdots&\vdots&\vdots&\vdots\\
0&0&0&\cdots&1&0\\
\zeta_{1}(\mbt)&\zeta_{2}(\mbt)&\zeta_{3}(\mbt)&\cdots&\zeta_{m-1}(\mbt)&f_m(\mbt)
\end{pmatrix},
\end{equation}
where the last row of $F$ is the same as the last row of $[Q_{m-1}]_{m\tm m}$.
We have   $\zeta_{i}\in L^2(\bT^N)$, $1\leq i\leq m-1$, and $f_m\in \bH_O^2(\bT^N_1)$.

Suppose
\begin{gather*}
f_m(t_1,t_2,\ldots,t_N)=\sum\nolimits_{0}^\iy t_1^k \gm_{k}(t_2,\ldots,t_N),\;\text{ where } \;\gm_{k}\in L^2(\bT^{N-1}),
\\
\zeta_i(t_1,t_2,\ldots,t_N)=\sum\nolimits_{k=-\iy}^\iy t_1^k \al_{i,k}(t_2,\ldots,t_N),\;\text{ where } \;\al_{i,k}\in L^2(\bT^{N-1}),
\end{gather*}
and 
$$
\zeta_{+,i}(t_1,\cdot)=\sum\nolimits_{k=0}^\iy  t_1^k \al_{i,k}(\cdot)\;
\text{ and }\; 
\zeta_{-,i}(t_1,\cdot)=\sum\nolimits_{k=-\iy}^{-1}  t_1^k \al_{i,k}(\cdot).
$$

{\bf Step 2.} Decompose $F$ as
$$
F(\mbt)=F_+(\mbt)F_-(\mbt),
$$
where $F_+$ and $F_-$ have the same structure as $F$ while their last rows are replaced by
$$
[\zeta_{+,1},\zeta_{+,2},\ldots, \zeta_{+,(m-1)},1]\;\text{ and } \; 
[\zeta_{-,1},\zeta_{-,2},\ldots, \zeta_{-,(m-1)},f_m],
$$
respectively.

{\bf Step 3.} For a sufficiently large $n$, approximate the matrix function $F_-$ by $F_-^{\{n\}}$ of the same structure
as  \eqref{7F} but with the last row  replaced by
$$
[\zeta_{-,1}^{\{n\}},\zeta_{-,2}^{\{n\}},\ldots, \zeta_{-,(m-1)}^{\{n\}},f_m^{\{n\}}],
$$
where
$$
\zeta_{-,i}^{\{n\}}(t_1,\cdot)=\sum\nolimits_{k=-n}^{-1}  t_1^k \al_{i,k}(\cdot)\;\text{ and } \;
f_{m}^{\{n\}}(t_1,\cdot)=\sum\nolimits_{k=0}^{n}  t_1^k \gm_{k}(\cdot).
$$

{\bf Step 4.} For the matrix function $F_-^{\{n\}}$, construct the corresponding unitary matrix function $U_m^{\{n\}}$ of the form \eqref {67.6}, where $u_{ij}\in\calP_+^n(\bT^N_1)$ (see \eqref{06.04.01}), which satisfies \eqref {67.7} such that
$$
F_-^{\{n\}} U_m^{\{n\}}\in \big(\calP_+^n(\bT^N_1)\big)^{m\tm m}.
$$
This construction can be realized pointwise for a.e. $(t_2,t_3,\cdots,t_N)\in\bT^{N-1}$ by the corresponding 1-D theorem proved in \cite{IEEE} (see Theorem 1 and its proof therein).

The matrix function $U_m$ is obtained as a limit of $ U_m^{\{n\}}$, as $n\to\iy$. The convergent properties of the algorithm is analyzed in  Sections 7.

As it is done in \cite{IEEE}, for a simplicity of the presentation, we can assume that each $\xi_{ij}$ in \eqref {7.16} is approximated by 
\begin{equation}\label{5.76}
\xi_{ij}(\mbt)\approx  \xi_{ij}^{\{(i-j)n\}}(\mbt)=
\sum\nolimits_{k=-(i-j)n}^\iy t_1^k C_{1k}\{\xi_{ij}\}(t_2,\ldots,t_N),
\end{equation}
and then we get the approximation of \eqref{7.17}
\begin{equation}\label{7.19}
S_{+,1}^{\{n\}}(\mbt)=M_1^{\{n\}}(\mbt)\mbU_2^{\{n\}}(\mbt)
\mbU_3^{\{2n\}}(\mbt)\ldots \mbU_r^{\{(r-1)n\}}(\mbt)=:M_1^{\{n\}}(\mbt) \calU^{\{n\}}(\mbt),
\end{equation} 
where $M_1^{\{n\}}$ is obtained from \eqref{7.16} by making the approximations \eqref{5.76} and each $ \mbU_m^{\{(m-1)n\}}$ is constructed according to the above described procedures taking $(m-1)n$ instead of $n$ in Step 3.
We recall that
$\det \calU^{\{n\}}= 1$ a.e. and, therefore, $$S_{+,1}^{\{n\}}\in\bH^2(\bT^N_1)^{d\tm d}_O\; \text{ for all }\; \;n=1,2,\ldots$$

The convergence  
\begin{equation}\label{14.02.01}
\|S_{+,1}^{\{n\}}-S_{+,1}\|_{L^2(\bT^N)}\to 0 \;\text{ as } n\to \iy
\end{equation} 
 is proved in Section 7.

	\section{ Description of the multivariable MSF method}

In this section, we outline a general scheme of the proposed N-D MSF method which can be utilized as a computational algorithm. 

Recursively with respect to $l$, we factorize the matrix \eqref{S} as
\begin{equation}\label{S+n}
S(t)=S_{+,l}(t)S_{+,l}^*(t),
\end{equation} 
$l=1,2,\ldots,N$, where 
\begin{equation}\label{6.15}
S_{+,l}\in\bH^2(\bT^N_l)_O^{d\times d}.
\end{equation} 
By virtue of  Definitions \ref{def2.4}, factorization \eqref {HL-MSF} is achieved as soon as we reach $l = N$, i.e.
$$
S_+=S_{+,N}\,.
$$

The basic procedure is the factorization
\begin{equation}\label{6.3}
S(t)=S_{+,1}(t)S_{+,1}^*(t)
\end{equation}
described in the previous section
(for uniqueness purposes, we assume that $S_{+,1}(0,t_2,\ldots,t_N)$  is positive definite in  \eqref{6.3}, however, it does not play any role).

For a factorable $S\in L^1(\bT^N)^{d\times d}$, where $N\geq 2$, we denote the resulting  factor $S_{+,1}\in\bH^2(\bT^N_1)_O^{d\times d}$ by
\begin{equation}\label{6.22}
S_{+,1}=:\calS\calF_{N}[S].
\end{equation}
In addition, we define the matrix function 
\begin{equation}\label{6.6}
\hat{S}_{+,1}\in L^2(\bT^{N-1})^{d\times d}
\end{equation}
for a.e. $(t_2,\ldots,t_N)\in \bT^{N-1}$ by the equation (cf. Remark \ref{remDef2.1})
\begin{equation}\label{6.35}
\hat{S}_{+,1}(t_2,\ldots,t_N)=S_{+,1}(0,t_2,\ldots,t_N)=\frac{1}{2\pi}\int_{\bT}S_{+,1}(t,t_2,\ldots,t_N)\,dt.
\end{equation}
Relation \eqref{6.6} holds because of \eqref{6.4} and the Fubini theorem. We also have
\begin{gather}
\int_{\bT^N}\log|\det S_{+,1}|\,d\mu_1=
\frac{1}{(2\pi)^N}\int_{\bT^{N-1}}\left(\int_{\bT}\log|\det S_{+,1}(t_1,t_2,\ldots,t_N)|\,dt_1\right)\,dt_2\ldots dt_N \notag\\ \label{6.45}
=\frac{1}{(2\pi)^N}\int_{\bT^{N-1}}\log|\det S_{+,1}(0,t_2,\ldots,t_N)|\,dt_2\ldots dt_N=
\int_{\bT^{N-1}}\log|\det\hat{S}_{+,1}|\,d\mu_{N-1}
\end{gather}
(the second equality holds due to \eqref{outer2}) which, together with \eqref{HL-PW} and \eqref{6.3}, implies that
\begin{equation}\label{6.8}
\log |\det \hat{S}_{+,1}|\in L^1({\mathbb T^{N-1}}).
\end{equation}
Applying the operator $\calS\calF_{N}$ defined by \eqref{6.22}, we proceed with factorization \eqref{S+n} as follows:
The relations  \eqref{6.6} and \eqref{6.8} imply that 
\begin{equation}\label{4.65}
S_1:=\hat{S}_{+,1}\hat{S}_{+,1}^*\in L^1(\bT^{N-1})^{d\times d}
\end{equation}
is a factorable matrix function and the operator $\calS\calF_{(N-1)}$ can be applied to it. Consequently, $\calU_2:=\hat{S}_{+,1}^{-1}\calS\calF_{(N-1)}[S_1]$ is a unitary matrix function and if we  define
\begin{equation}\label{4.67}
S_{+,2} (t_1,t_2,\ldots,t_N):= S_{+,1} (t_1,t_2,\ldots,t_N)\,\calU_2(t_2,\ldots,t_N),
\end{equation}
we get
$$
S(t)=S_{+,2}(t)S_{+,2}^*(t),
\;\text{ and }\:
S_{+,2}\in\bH^2(\bT^N_2)_O^{d\times d}.
$$
Similarly, if the factorization
$$
S(t)=S_{+,l-1}(t)S_{+,l-1}^*(t),
$$
has already been constructed, 
where $ S_{+,l-1}\in\bH^2(\bT^N_{l-1})_O^{d\times d}$, then we define
\begin{equation}\label{6.87}
\hat{S}_{+,l-1}(t_l,\ldots,t_N)=
S_{+,l-1}(0,\ldots,0,t_l,\ldots,t_N)
=\int_{\bT^{l-1}}S_{+,l-1}(\cdot,t_l,\ldots,t_N)\,d\mu_{l-1}
\end{equation}
and
\begin{equation}\label{6.88}
S_{l-1}=\hat{S}_{+,l-1}\hat{S}_{+,l-1}^*\,,
\end{equation}
apply the operator $\calS\calF_{(N-l+1)}$ to $S_{l-1}$ to get a unitary matrix function
\begin{equation}\label{6.121}
\calU_l(t_l,\ldots,t_N):=\hat{S}_{+,l-1}^{-1}(t_l,\ldots,t_N)  \calS\calF_{(N-l+1)}[S_{l-1}](t_l,\ldots,t_N),
\end{equation}
and obtain 
\begin{equation}\label{6.117}
S_{+,l} (t_1,\ldots,t_N)= S_{+,l-1} (t_1,\ldots,t_N)\,\calU_l(t_l,\ldots,t_N)
\end{equation}
which satisfies \eqref{S+n} and \eqref{6.15}.

Summarizing, we get
\begin{equation}\label{12.02.01}
S_+=S_{+,1}\,\calU_2\, \calU_3\,\ldots \,\calU_N.
\end{equation}

In order to satisfy the uniqueness condition, we can take the spectral factor which is positive definite at the origin
$$
S_+(z)\big(S_+(\mbo)\big)^{-1}\sqrt{S_+(\mbo)(S_+(\mbo))^*}.
$$

Note also that 
\begin{equation}\label{6.89}
\hat{S}_{+,l}(t_{l+1},\ldots,t_N)= \hat{S}_{+,l-1}(0,t_{l+1},\ldots,t_N).
\end{equation}

In actual computations, the equations presented in this section are changed with approximations. For sufficiently large positive integers $n_1, n_2,\ldots,n_N$, we proceed as follows: 
First we approximate $S_{+,1}$ as it is described in Procedure 3 of the previous section 
\begin{equation}\label{6.82}
S_{+,1}\approx S_{+,1}^{\{n_1\}}=: \calS\calF_{N}^{\{n_1\}}[S].
\end{equation}
Then we compute $\hat{S}_{+,1}^{\{n_1\}}(t_2,\ldots,t_N)=S_{+,1}^{\{n_1\}}
(0,t_2,\ldots,t_N)$,  approximate 
$$
S_1\approx S_1^{\{n_1\}}=\hat{S}_{+,1}^{\{n_1\}}(\hat{S}_{+,1}^{\{n_1\}})^*,
$$
and obtain its approximate factor $\calS\calF_{(N-1)}^{\{n_2\}}[S_1^{\{n_1\}}]$ which gives an approximation of $S_{+,2}$ as
(cf. \eqref{4.67})
$$
S_{+,2}\approx S_{+,2}^{\{n_1n_2\}}=  S_{+,1}^{\{n_1\}}
(\hat{S}_{+,1}^{\{n_1\}})^{-1} \calS\calF_{(N-1)}^{\{n_2\}}[S_1^{\{n_1\}}]=:  S_{+,1}^{\{n_1\}}\calU_2^{\{n_1n_2\}}.
$$
Continuing in this manner, we obtain
$$
S_+=S_{+,N}\approx S_{+,N}^{\{n_1\ldots n_N\}}= S_{+,1}^{\{n_1\}}\calU_2^{\{n_1n_2\}}\ldots \calU_N^{\{n_1n_2\ldots n_N\}},
$$
where the following functions are defined recursively 
\begin{gather}\label{15.02.04}
\calU_l^{\{n_1n_2\ldots n_l\}}=(\hat{S}_{+,l-1}^{\{n_1\ldots n_{l-1}\}})^{-1}
\calS\calF_{(N-l+1)}^{\{n_l\}}[S_{l-1}^{\{n_1\ldots n_{l-1}\}}],\\
\label{6.12n}
\hat{S}_{+,l-1}^{\{n_1\ldots n_{l-1}\}}(t_l,\ldots,t_N)=S_{+,l-1}^{\{n_1\ldots n_{l-1}\}}(0,\ldots,0,t_l,\ldots,t_N)=\int_{T^{l-1}}
S_{+,l-1}^{\{n_1\ldots n_{l-1}\}}(\cdot,t_l,\ldots,t_N)\,d\mu_{l-1}\\
\label{6.13n}
 \;\;S_{l-1}^{\{n_1\ldots n_{l-1}\}}=
\hat{S}_{+,l-1}^{\{n_1\ldots n_{l-1}\}}\big(\hat{S}_{+,l-1}^{\{n_1\ldots n_{l-1}\}}\big)^* ,\\\label{6.185}
 S_{+,l}^{\{n_1n_2\ldots n_l\}}=
S_{+,l-1}^{\{n_1\ldots n_{l-1}\}}\, \calU_l^{\{n_1n_2\ldots n_l\}}.
\end{gather}

It is proved in the next section that 
\begin{equation}\label{6conv}
S_{+,N}^{\{n_1\ldots n_N\}}\to S_+  \;\text{ in } L^2(\bT^N)
\end{equation}
as all integers $n_1,n_2,\ldots,n_N\to \iy$

\section{Convergence properties of the algorithm}
In this section, we show that although the $N$-D MSF algorithm is described ``pointwise", it has convergence properties globally. Namely,  we prove convergences \eqref{14.02.01} and \eqref {6conv}. To this end, we need to introduce some additional definitions and prove auxiliary statements. 

 We say that a sequence $f_n\in L^2(\bT^N)$, $n=1,2,\ldots$, is convergent to $f\in L^p(\bT^N)$, $p\geq 1$, ``restricted to hyperplanes”, denoted
$$
f_n\rightarrowtail f \;\;\text{ in } L^p(\bT^N),
$$
if for each $l=2,3,\ldots,N$ and for a.e. $(t_l,t_{l+1},\ldots,t_N)\in\bT^{N-l+1}$,
$$
f_n(\cdot,t_l,\ldots,t_N)\to f(\cdot,t_l,\ldots,t_N)\;\;\text{ in } L^p(\bT^{l-1}) 
$$
as functions of variables $(t_1,t_2,\ldots,t_{l-1})$. It is assumed that $f_n\to f$ in $L^p(\bT^N)$ as well.

\begin{remark}
Simple examples show that, in general, $f_n\to f$ $\not\Longrightarrow $ $f_n\rightarrowtail f$ $($in the same $L^p(\bT^N)${}$)$. This happens because  convergence in norm does not imply  convergence almost everywhere. 
\end{remark}

For convenience of references, we prove some lemmas which easily follow from the Fubini theorem.

\begin{lemma}
	Let  $f\in L^2(\bT^N)$, and let 
	$$
	f_n(t_1,t_2,\ldots,t_N)=\sum\nolimits_{k=-n}^n t_1^k  C_{1k}\{f\}(t_2,\ldots,t_N).
	$$
	Then 
	\begin{equation}\label{14.02.03}
	f_n\rightarrowtail f \;\;\text{ in } L^2(\bT^N).
		\end{equation}
\end{lemma}
\begin{proof}
	For a.a. $(t_l,\ldots,t_N)\in \bT^{N-l+1}$, we have 
$$
f_{t_l\ldots t_N}:=f(\cdot,t_l,\ldots,t_N)\in L^2(\bT^{l-1}) 
$$
and
$$
C_{1k}\{f\}(t_2,\ldots,t_N)=C_{1k}\{f_{t_l\ldots t_N}\}(t_2,\ldots,t_{l-1})
\;\text{ for  a.a. }\; (t_2,\ldots,t_{l-1})\in\bT^{l-2}.
$$
Therefore, for a.a. $(t_l,\ldots,t_N)\in \bT^{N-l+1}$,
\begin{gather*}
\int_{\bT^{l-1}}\left| f_{t_l\ldots t_N}(t_1,\ldots,t_{l-1})-
\sum\nolimits_{k=-n}^n  t_1^k  C_{1k}\{f\}(t_2,\ldots,t_N)\right|^2\,dt_1\ldots dt_{l-1}=\\
\int_{\bT^{l-2}}\left(\int_{\bT}\left| f_{t_l\ldots t_N}(t_1,\ldots,t_{l-1})-
\sum\nolimits_{k=-n}^n  t_1^k  C_{1k}\{f_{t_l\ldots t_N}\}(t_2,\ldots,t_{l-1})\right|^2\,dt_1\right)\,dt_2\ldots dt_{l-1}\to 0
\end{gather*}
since the second integral in the last expression converges to $0$ for a.a. $ (t_2,\ldots,t_{l-1})\in\bT^{l-2}$ (due to pointwise application of the Parseval’s identity) and it is majorized by
$$
\int_{\bT}\left| f_{t_l\ldots t_N}(t_1,\ldots,t_{l-1})\right|^2\,dt_1\in 
L^2(\bT^{l-2}).
$$
Hence the convergence ``restricted to hyperplanes” \eqref{14.02.03} follows.
	\end{proof}

For a function $f\in L^1(\bT^N)$ and $1\leq j\leq N-1$, 
slightly abusing the notation (cf. \eqref {2.81}),
let $\hat{f}\in L^1(\bT^{N-j})$ be the function defined by
$$
\hat{f}(t_{j+1},\ldots,t_N)=\int_{\bT^j} f(\cdot,t_{j+1},\ldots,t_N)\,d\mu_j.
$$

\begin{lemma}  Let 
	$$
	f_n\rightarrowtail f \;\;\text{ in } L^p(\bT^N),
	$$
where $p\geq 1$.	Then, for each $j=1,2,\ldots,N-1$,
	\begin{equation}\label{14.02.04}
	\hat{f}_n\rightarrowtail \hat{f} \;\;\text{ in } L^p(\bT^{N-j}).
		\end{equation}
\end{lemma}
\begin{proof} Let $j+1<l\leq N$. For a.a. $(t_l,t_{l+1},\ldots,t_N)\in\bT^{N-l+1}$, we have
	\begin{gather*}
	\big\|\hat{f}_n(\cdot,t_{l},\ldots,t_N)-\hat{f}(\cdot,t_{l},\ldots,t_N)\big\|^p_{L^p(\bT^{l-j-1})}\\
=\left(\frac{1}{2\pi}\right)^{l-j-1}	\int_{\bT^{l-j-1}} \left| \int_{\bT^j} f_n(\cdot,t_{j+1},\ldots,t_N)\,d\mu_j-
 \int_{\bT^j} f(\cdot,t_{j+1},\ldots,t_N)\,d\mu_j
\right|^p\,dt_{j+1}\cdots\,dt_{l-1} \\
\leq \left(\frac{1}{2\pi}\right)^{l-j-1}
	\int_{\bT^{l-j-1}} \left( \int_{\bT^j}\big| f_n(\cdot,t_{j+1},\ldots,t_N)- f(\cdot,t_{j+1},\ldots,t_N)\big|^p\,d\mu_j\right) \,dt_{j+1}\cdots\,dt_{l-1}\\	=\big\|f_n(\cdot,t_{l},\ldots,t_N)-f(\cdot,t_{l},\ldots,t_N)\big\|_{L^1(\bT^{l-1})}^p\to 0\,,
	\end{gather*}
	which implies \eqref{14.02.04}.
\end{proof}

The basic step in the proof of the convergence is the following 
\begin{theorem}
		Let $S^{\{n\}}$, $n=1,2,\ldots,$ and $S$ satisfy \eqref{S},
	\begin{equation}\label{21.2.1}
	S^{\{n\}}\rightarrowtail S  \;\text{ in }\; L^1(\bT^N)
	\end{equation}
	and, for a.a. $(t_2,t_3,\ldots,t_N)\in\bT^{N-1}$,
	$$
	\int_{\bT}\log\det S^{\{n\}}(t_1,t_2,\ldots,t_N)\,dt_1\to 
	\int_{\bT}\log\det S(t_1,t_2,\ldots,t_N)\,dt_1.
	$$
	Suppose 
	\begin{equation}\label{21.2.2}
	S^{\{n\}}(\mbt)=S^{\{n\}}_{+,1}(\mbt)\left(S^{\{n\}}_{+,1}(\mbt)\right)^*
	\end{equation}
	is the factorization of $S^{\{n\}}$ defined according to Section 5. Then
	\begin{equation}\label{21.2.3}
	S^{\{n\}}_{+,1} \rightarrowtail S_{+,1}  \;\text{ in }\; L^2(\bT^N).
	\end{equation}
\end{theorem}
\begin{proof}
	By virtue of 1-dimensional Theorem \ref {th.2.1}, for a.a. $(t_2,t_3,\ldots,t_N)\in\bT^{N-1}$, we have
	$$
	S^{\{n\}}_{+,1}(\cdot, t_2,\ldots,t_N)\to
	S_{+,1}(\cdot, t_2,\ldots,t_N)\;\text{ in } L^2(\bT).
	$$
	This implies  the convergence in measure for each $l=2,\ldots,N$ and  a.a. $(t_l,\ldots,t_N)\in\bT^{N-l+1}$:
	$$
	S^{\{n\}}_{+,1}(\cdot, t_l,\ldots,t_N) \rightrightarrows
	S_{+,1}(\cdot, t_l,\ldots,t_N)\;\text{ on } \bT^{l-1}.
	$$
	Equation \eqref{21.2.2} guarantees that the squares of absolute values of the entries of matrix functions $S^{\{n\}}_{+,1}(\cdot, t_l,\ldots,t_N)$ are bounded by diagonal entries of $S^{\{n\}}(\cdot, t_l,\ldots,t_N)$ which are convergent in $L^1(\bT^{l-1})$ by virtue of the definition of the convergence (``restricted to hyperplanes”) in \eqref {21.2.1}.	
	Therefore, Proposition 2.2 implies that 
	$$
	S^{\{n\}}_{+,1}(\cdot, t_l,\ldots,t_N) \to
	S_{+,1}(\cdot, t_l,\ldots,t_N)\;\text{ in } L^1(\bT^{l-1}).
	$$
	Hence \eqref{21.2.3} holds.
\end{proof}

Theorem 7.1 implies the convergence 
	\begin{equation}\label{15.02.00}
S^{\{n_1\}}_{+,1} \rightarrowtail S_{+,1}  \;\text{ in }\; L^2(\bT^N)
\end{equation}
for $S^{\{n_1\}}_{+,1}$  defined by \eqref {7.19},
which in particular contains \eqref{14.02.01}. Indeed, $S^{\{n_1\}}_{+,1}$  is a spectral factor of $M_1^{\{n_1\}}( M_1^{\{n_1\}})^*$ which can be taken in the role of $S^{\{n\}}$ in Theorem 7.1. By virtue of Lemma 7.1 and H\"{o}lder's inequality, we have
$$
M_1^{\{n_1\}}( M_1^{\{n_1\}})^* \rightarrowtail M_1M_1^*=S \; \text{ in }\; L^1(\bT),
$$
and determinants are also equal for each $n_1$
$$
\det \left(S^{\{n_1\}}_{+,1}(S^{\{n_1\}}_{+,1})^* \right)=\det \left( M_1M_1^*\right)=\det S
$$
because of the structure of matrices $S^{\{n_1\}}_{+,1}$ and $M_1^{\{n_1\}}$, see \eqref{7.19}. Hence, the hypothesis of Theorem 7.1 are satisfied and \eqref{15.02.00} holds.

The important observation is that as the factorization proceeds the determinants of the obtained matrices remain unchanged, namely
	\begin{equation}\label{15.02.01}
\det S^{\{n_1\ldots n_l\}}_l=\det S_l\;\text{ for each } l=2,\ldots,N.
\end{equation}
 This  can be justified recursively by using the definitions \eqref {6.22}, \eqref {6.82}, \eqref {6.87}, \eqref{6.88}, \eqref{6.12n},  \eqref{6.13n}, and the property \eqref{6.89}:
\begin{gather*}
\det S^{\{n_1\ldots n_{l-1}\}}_{l-1}=\det S_{l-1}\Rightarrow
\left| \det \left( \calS\calF^{\{n_l\}}_{(N-l+1)}[ S^{\{n_1\ldots n_{l-1}\}}_{l-1}]\right)\right|=\left| \det \left( \calS\calF_{(N-l+1)}[ S_{l-1}]\right)\right| 
\\ 
\text{(see Remark \ref{rem2.3} and also \eqref{6.45})} \Rightarrow
\left| \det \hat{S}^{\{n_1\ldots n_{l}\}}_{+,l}\right|=\left|\det \hat{S}_{+,l}\right| \Rightarrow 
  \det {S}^{\{n_1\ldots n_{l}\}}_{l}=\det {S}_{l}.
\end{gather*}

We are now ready  to show 
	\begin{equation}\label{15.02.02}
S_{+,N}^{\{n_1\ldots n_N\}} \rightarrowtail S_{+,N}=S_+  \;\text{ in } L^2(\bT^N),
\end{equation}
and consequently \eqref {6conv}, by using again the recursive steps. The convergence \eqref{15.02.00} can be used as a starting point. Now assume that
$$
S_{+,l-1}^{\{n_1\ldots n_{l-1}\}} \rightarrowtail S_{+,l-1}  \;\text{ in } L^2(\bT^N).
$$
This implies that (see \eqref{6.87} and \eqref{6.12n})
$$
\hat{S}_{+,l-1}^{\{n_1\ldots n_{l-1}\}} \rightarrowtail \hat{S}_{+,l-1}  \;\text{ in } L^2(\bT^{N-l+1})
$$
by virtue of Lemma 7.2, which in turn implies that (see \eqref{6.88} and \eqref{6.13n})
$$
S_{l-1}^{\{n_1\ldots n_{l-1}\}} \rightarrowtail S_{l-1} \;\text{ in } L^1(\bT^{N-l+1})
$$
because of  H\"{o}lder's inequality. Combining this with \eqref{15.02.01}, we obtain
$$
 \calS\calF^{\{n_l\}}_{(N-l+1)}[ S^{\{n_1\ldots n_{l-1}\}}_{l-1}]
  \rightarrowtail \calS\calF_{(N-l+1)}[ S_{l-1}]  \;\text{ in } L^2(\bT^{N-l+1}).
$$
by virtue of Theorem 7.1. Consequently (see \eqref{6.121} and \eqref{15.02.04}), $\calU_l^{\{n_1n_2\ldots n_l\}}\rightrightarrows \calU_l$  and applying Corollary 2.2 we get (see \eqref{6.117} and \eqref{6.185})
$$
{S}_{+,l}^{\{n_1\ldots n_{l}\}}={S}_{+,l-1}^{\{n_1\ldots n_{l-1}\}} \calU_l^{\{n_1n_2\ldots n_l\}} \rightarrowtail {S}_{+,l}  \;\text{ in } L^2(\bT^{N}).
$$

\section{Numerical simulations}

The computer code for numerical testing of the proposed algorithm was written in \mbox{MATLAB}. A simple outer  type  polynomial matrix 
$$
A=\begin{pmatrix}
4+y-xy^{-1}+x+2xy& 1+2y+x+xy\\ 1+y+2xy^{-1}+2x+2y& 5+y+xy^{-1}-x+xy
\end{pmatrix}
$$
(we relabel  $x=t_1$ and $y=t_2$; the integer coefficients are also chosen for notational simplicity) was designed, which is positive definite at the origin,  and $S$ was constructed as
\begin{equation}\label{24.8.1}
S=AA^*.
\end{equation} 
Then \eqref{24.8.1} is an exact spectral factorization of 
$$
S=\begin{pmatrix} S_{11}  & S_{12}\\ S_{21}&  S_{22}
\end{pmatrix},
$$
where
$
S_{11}=\newline
9x^{-1}y^{-1}     +9 x^{-1}   -x^{-1}y    -x^{-1}y^2
-2 y^{-2}  +   8 y^{-1}  +   30   +   8 y   -2 y^2
-xy^{-2}    -xy^{-1}   +  9 x  +  9 xy$;

$
S_{12}= \newline
9x^{-1}y^{-1}     +11 x^{-1}   +9x^{-1}y    +4x^{-1}y^2
-2 y^{-2}  +   6y^{-1}  +   16   +   17y   +5y^2
-xy^{-2}    +xy^{-1}   +  9 x  +  7 xy$;

$
S_{21}=\newline
7x^{-1}y^{-1}     +9x^{-1}   +x^{-1}y    -x^{-1}y^2
+5 y^{-2}  +   17y^{-1}  +   16   +   6y  -2y^2
+4xy^{-2}    +9xy^{-1}   + 11 x  +  9 xy $;

$
S_{22}     =
7x^{-1}y^{-1}   +8x^{-1}y    +3x^{-1}y^2
+5 y^{-2}  +   12y^{-1}  +   43  +   12y  +5y^2
+3xy^{-2}    +8xy^{-1}   + 7xy;       
$\newline

In less than 2 seconds, a complete 16 digit accuracy of MatLab double precision has been achieved on a computer with the following characteristics: Intel(R) Core(TM) i7 8650U CPU, 1.90 GHz, RAM 16.00 Gb.

As was mentioned in the Introduction, there exist several  multivariable MSF methods available  in the literature. However, none of them provide numerical examples of factorized multivariable matrices. Therefore, we were unable to carry out a comparative analysis of the proposed method based on numerical simulations.

\section{Appendix: Application to Granger causality}

Granger causality \cite{Granger} has emerged in recent years as one of the leading statistical techniques in neuroscience for inferring directions of neural interactions and information flow in the brain from collected multidimensional data. In this Appendix, for illustrative purposes, we demonstrate the application of spectral factorization in Granger causality. The basic idea can be traced back to Wiener \cite{Wiener}.

 For two jointly stationary processes $\ldots, X_{-1},X_0,X_1,X_2\ldots$ and $\ldots, Y_{-1},Y_0,Y_1,Y_2\ldots$, let
$$
X_{n+1}=\sum\nolimits_{k=0}^\iy a_k X_{n-k}+\ep_n
$$
be the autoregressive representation of the process $X$, and let
$$
X_{n+1}=\sum\nolimits_{k=0}^\iy b_k X_{n-k}+\sum\nolimits_{k=0}^\iy c_k Y_{n-k}+\eta_n
$$
be its joint representation, where $\ep$ and $\eta$ are corresponding noise terms. It is assumed that $X_k, Y_k$ belong to a Hilbert space $\calH$ and
\begin{equation}\label{GC3}
\langle Z^i_n, Z^j_{n+k}\rangle=\int_{\bT} t^{-k}\,d\nu_{ij},
\end{equation}
where $i,j=1,2$; $Z^1=X$, $Z^2=Y$; and $\nu=(\nu_{ij})$ is a matrix spectral measure defined on $\bT$. 
Therefore, the value $\sigma_1=\|\ep_n\|$ measures the accuracy of the autoregressive prediction of $X_n$ based on its previous values, whereas the value $\Sigma_1=\|\eta_n\|$ represents the
accuracy of predicting the same value  $X_n$ based on the previous values of both $X$ and $Y$.
According to Wiener \cite{Wiener} and Granger \cite{Granger}, if $\Sigma_1$ is less than $\sigma_1$ in some suitable statistical sense, then $Y$ is said to have a causal 
influence on $X$. This causal influence is quantified by (see \cite{Dhamala})
\begin{equation}\label{GC5}
F_{Y\to X}= \ln\frac{\sigma_1}{\Sigma_1}.
\end{equation}	

Recently the concept of multi-step Granger causality has also been introduced \cite{MGC} with
\begin{equation}\label{GC6}
F^L_{Y\to X}= \ln\frac{\sigma_L}{\Sigma_L},
\end{equation}	
where 
\begin{equation}\label{GC7}
\sigma_L=\inf_{a_k} \left\|X_{n+L}-\sum\nolimits_{k=0}^\iy a_k X_{n-k}\right\|
\end{equation}
and
\begin{equation*}\label{GC8}
\Sigma_L=\inf_{b_k, c_k} \left\|X_{n+L}-\sum\nolimits_{k=0}^\iy b_k X_{n-k}
-\sum\nolimits_{k=0}^\iy c_k Y_{n-k}\right\|.
\end{equation*}  	
As usually is the case in applications, let us assume below  that the stationary processes are regular and non-deterministic. Therefore, the spectral measure is absolutely continuous 
$$
\nu(t)=S(t)\,dt
$$
and the matrix function $S$ satisfies the factorability condition \eqref{PW} (see \cite{Roz}).

By virtue of \eqref{GC3}, the process $\ldots, X_{-1},X_0,X_1,X_2\ldots$ is unitary equivalent to $\{t^n\}_{n\in\bZ}$ in the Hilbert space $L^2(d\nu_{11})$. Therefore, if $\nu_{11}(t)=f(t)\,dt$ and $f(t)=f_+(t)\ol {f_+(t)}$ is the spectral factorization of $f$, then \eqref{GC7} can be expressed by
\begin{equation}\label{GC77}
\sigma_L=\sqrt{\sum\nolimits_{k=0}^{L-1}|C_k\{f_+\}|^2}=\big\|\bP_{L-1}[f_+]\big\|,
\end{equation} 
where $\bP_L$ stands for the projection operator acting as $\bP_L: \sum\nolimits_{k=0}^{\iy} c_kt^k\to \sum\nolimits_{k=0}^{L} c_kt^k$.  
Indeed, 
\begin{gather*}
\sigma_L^2=\inf_{a_k}
\big\|x_{n+L}-\sum\nolimits_{k=0}^\iy a_k x_{n-k}\big\|^2=
\inf_{a_k}\frac1{2\pi}\int_{T} \left|t^{-L}-\sum\nolimits_{k=0}^\iy a_kt^k\right|^2 f(t)\,dt\\
=	\inf_{a_k}\frac1{2\pi}\int_{T} \left(t^{-L}-\sum\nolimits_{k=0}^\iy a_kt^k\right)f_+(t)
\ol{\left(t^{-L}-\sum\nolimits_{k=0}^\iy a_kt^k\right)f_+(t)}\,dt\\
=\inf_{a_k}\left\|t^{-L} \bP_{L-1}[f_+]+ \sum\nolimits_{k=0}^{\iy}C_{k+L}\{f_+\}t^k-\sum\nolimits_{k=0}^\iy a_kt^k f_+(t) \right\|^2_2\\
=\big\|t^{-L}\bP_{L-1}[f_+]\big\|^2_2+\inf_{a_k}\left\| \sum\nolimits_{k=0}^{\iy}C_{k+L}\{f_+\}t^k-\sum\nolimits_{k=0}^\iy a_kt^k f_+(t) \right\|^2_2=\big\|\bP_{L-1}[f_+]\big\|^2_2,
\end{gather*}
 The infimum in the previous line  is equal to zero because of the Beurling theorem (see, e.g. \cite{Koo}).

This reasoning can be extended to the matrix case, and it can be proved that 
\begin{equation}\label{GC10}
\Sigma_L=\sqrt{\sum\nolimits_{k=0}^{L-1} \big(|C_k\{S_{11}^+\}|^2+|C_k\{S_{12}^+\}|^2\big)},
\end{equation}  
where $S^+_{11}$ and $S^+_{12}$ are the corresponding entries in  the spectral factor of $S_+$. Indeed, using the notation $\|(a,b)\|_2^2=\|a\|_2^2+\|b\|_2^2$ for $a$ and $b$ from $L^2(\bT)$ and simirarily to the scalar case, we get in the matrix case (cf. \cite[7.9]{Wie57}):
\begin{gather*}
\Sigma_L^2=\inf_{N,\al,\bt}
\big\|x_{n+L}-\sum\nolimits_{k=0}^N \al_k x_{n-k}-\sum\nolimits_{k=0}^N \bt_k y_{n-k}\big\|^2.\\
=\inf_{N,\al_k,\bt_k}
\frac1{2\pi}\int_{\bT}\Big((t^{-L},0)-\sum\nolimits_{k=0}^N (\al_k,\bt_k)t^k\Big)S(t)
\Big((t^L,0)^T-\sum\nolimits_{k=0}^N (\al_k,\bt_k)^*t^{-k}\Big)\,dt=\\ \inf_{N,\al_k,\bt_k}
\frac1{2\pi}\int_{\bT}\Big((t^{-L},0)-\sum_{k=0}^N (\al_k,\bt_k)t^k\Big)S^+(t)\big(S^+(t)\big)^*
\Big((t^L,0)^T-\sum_{k=0}^N (\al_k,\bt_k)^*t^{-k}\Big)\,dt=\\ \inf_{N,\al_k,\bt_k}
\left\|t^{-L}\big(\bP_{L-1}[S^+_{11}],\bP_{L-1}[S^+_{12}]\big)+
\sum_{k=0}^{\iy}\Big(C_{k+L}\{S_{11}^+\},
C_{k+L}\{S_{12}^+\}\Big)t^k
-\sum_{k=0}^N (\al_k,\bt_k)t^kS^+(t)\right\|^2_2\\
=\left\|\big(\bP_{L-1}[S^+_{11}],\bP_{L-1}[S^+_{12}]\big)\right\|^2_2+
\inf_{N,\al_k,\bt_k}
\left\|\sum_{k=0}^{\iy}\Big(C_{k+L}\{S_{11}^+\},
C_{k+L}\{S_{12}^+\}\Big)t^k
-\sum_{k=0}^N (\al_k,\bt_k)t^kS^+(t)\right\|^2_2\\
=\|\bP_{L-1}[S^+_{11}]\|^2_2+\|\bP_{L-1}[S^+_{12}]\|^2_2.
\end{gather*}
Since $S^+$ is an outer analytic matrix function,  the last infimum is $0$  by the vector generalization of Beurling theorem  (see, e.g. \cite{EL10}).
Thus \eqref{GC10} holds.

Due to \eqref{GC6}, \eqref{GC77}, and \eqref{GC10}, we have
\begin{equation}\label{Ab44}
F^L_{Y\to X}=\ln\frac{\sum\nolimits_{k=0}^{L-1}|C_k\{f_+\}|^2}{\sum\nolimits_{k=0}^{L-1} \big(|C_k\{S_{11}^+\}|^2+|C_k\{S_{12}^+\}|^2\big)}.
\end{equation}
Consequently, the spectral factorization plays a crucial role in estimating the Granger causality.

The multivariable matrix spectral factorization algorithm proposed in this paper provides perspective to develop  the multivariate Granger causality theory for the situation where the collected data depends on more than one parameter (e.g. when we have a spatio-temporal dependence of random variables on indices). For simplicity of notation, we assume that the number of these parameters is two  and we have jointly stationary processes $X_{nm}$ and $Y_{nm}$.  Then there exists a spectral measure $\nu=(\nu_{ij})$ defined on $\bT^2$ such that
\begin{equation}\label{GC3}
\langle Z^i_{mn}, Z^j_{m+k,n+l}\rangle=C_{kl}=\int_{\bT} t_1^{-k}t_2^{-l}\,d\nu_{ij}(t_1,t_2),
\end{equation}
where $Z^1=X$ and $Z^2=Y$. Again we assume that $\nu$ is absolutely continuous: $\nu(\mbt)=S(\mbt)\,d\mbt$, and $S$ satisfies the factorability condition \eqref {HL-PW}. The half-plane $H_2\sbs\bZ^2$ imposes the ``temporal order” in $\bZ^2$ then, namely
$$(n,m)\geq (k,l)\Longleftrightarrow (n-k,m-l)\in H_2$$
and the corresponding Granger causality can be defined by
$$
F^{LM}_{Y\to X}= \ln\frac{\sigma_{LM}}{\Sigma_{LM}},
$$
where
\begin{equation}\label{GC11}
\sigma_{LM}=\inf_{a_{kl}}\left\|X_{n+L,m+M}-\sum\nolimits_{(k,l)\in H_2}a_{kl}X_{n-k,m-l}\right\|
\end{equation}
and 
\begin{equation}\label{GC12}
\Sigma_{LM}=\inf_{b_{kl},c_{kl}}\left\|X_{n+L,m+M}-\sum\nolimits_{(k,l)\in H_2}b_{kl}X_{n-k,m-l}-\sum\nolimits_{(k,l)\in H_2}c_{kl}Y_{n-k,m-l}\right\|.
\end{equation}
Denoting   the (multivariable) scalar spectral factor of $S_{11}$ by $f_+$ and  the (multivariable) matrix spectral factor of $S$ by $S^+$, one can show that 
$$
\sigma_{LM}^2=\sum_{(0,0)\leq (k,l)<(L,M)}| C_{kl}\{f_+\}|^2
$$
and
$$
\Sigma_{LM}^2=\sum\limits_{(0,0)\leq (k,l)<(L,M)}\left( |C_{kl}\{S^+_{11}\}|^2+ |C_{kl}\{S^+_{12}\}|^2\right),
$$
similarly to \eqref{GC77} and \eqref{GC10}.  Therefore, similarly to \eqref{Ab44}, we have
$$
F^{LM}_{Y\to X}= \ln\frac{\sum_{(0,0)\leq (k,l)<(L,M)}| C_{kl}\{f_+\}|^2}{\sum_{(0,0)\leq (k,l)<(L,M)}\left( |C_{kl}\{S^+_{11}\}|^2+ |C_{kl}\{S^+_{12}\}|^2\right)}.
$$

\section{Acknowledgments}

The authors were supported in part by Faculty Research funding
from the Division of Science and Mathematics, New York University Abu Dhabi.
The first author was also partially supported by the Shota Rustaveli National Science Foundation of Georgia
(Project No. FR-18-2499).

%		\bibliographystyle{amsplain}
%		\bibliography{Ref_Helson} 

\def\cprime{$'$}
\providecommand{\bysame}{\leavevmode\hbox to3em{\hrulefill}\thinspace}
\providecommand{\MR}{\relax\ifhmode\unskip\space\fi MR }
\providecommand{\MRhref}[2]{%
	\href{http://www.ams.org/mathscinet-getitem?mr=#1}{#2}
}
\providecommand{\href}[2]{#2}

\end{document}